\documentclass[12pt]{article}

\usepackage{amsmath, amsthm, amssymb, amscd}
\usepackage{bm, mathrsfs, dsfont}
\usepackage{bbding}
\usepackage{hyperref}
\usepackage[all,pdf]{xy}
\usepackage{upgreek}

\newtheorem{defn}{Definition}[section]
\newtheorem{qu}{Question}[section]
\newtheorem{conj}{Conjecture}[section]
\newtheorem{lem}{Lemma}[section]

\newtheorem{thm}{Theorem}[section]
\newtheorem{cor}{Corollary}[section]
\newtheorem{rem}{Remark}[section]

\title{\bf Cone spherical metrics and stable vector bundles}
\author{Lingguang Li, Jijian Song and Bin Xu}

\begin{document}
\maketitle


\noindent{\small {\bf Abstract:} {\it Cone spherical metrics} are conformal metrics with constant curvature one and finitely many conical singularities on compact Riemann surfaces. A cone spherical metric is called {\it irreducible} if each developing map of the metric does {\it not} have monodromy lying in ${\rm U(1)}$. We establish on compact Riemann surfaces of positive genera a correspondence between irreducible cone spherical metrics with cone angles being integral multiples of $2\pi$ and line subbundles of rank two stable vector bundles. Then we are motivated by it to prove a theorem of Lange-type that there always exists a stable extension of $L^*$ by $L$, for $L$ being a line bundle of negative degree on each compact Riemann surface of genus greater than one.  At last, as an application of these two results, we obtain a new class of irreducible spherical metrics with cone angles being integral multiples of $2\pi$  on each compact Riemann surface of genus greater than one}.\\



\section{Introduction}
Let $X$ be a compact connected Riemann surface of genus $g_X$ and $D = \sum_{j=1}^n \, \beta_j \, p_j$ an $\mathbb{R}$-divisor on $X$ such that $p_1, \cdots, p_n$ are distinct $n \geq 1$ points on $X$ and $0 \neq \beta_j > -1$. We call a smooth conformal metric $g$ on $X \setminus {\rm supp} \, D := X \setminus \{p_1, \cdots, p_n\}$ a {\it conformal metric representing $D$ on $X$} if and only if, for each point $p_j$, there exists a complex coordinate chart $(U, \, z)$ centered at $p_j$ such that the restriction of $g$ to $U \setminus \{ p_j \}$ has form $e^{2 \varphi } \, |dz|^2$, where the real valued function $\varphi - \beta_j \, \ln \, |z|$ extends to a continuous function on $U$. In other words,  $g$ {\it has a conical singularities at each $p_j$ with cone angle $2 \pi( 1 + \beta_j )$}. Under mild regularity assumption on the Gaussian curvature function $K_g$ of $g$, Troyanov proved in \cite[Proposition 1]{Troyanov91} the Gauss-Bonnet formula
\[ \frac{1}{2\pi}\int_{X\setminus {\rm supp}\, D}\, K_g\, {\rm d}A_g = 2-2g_X+\deg\, D, \]
where 
we denote by $\deg \, D = \sum_{j=1}^n \, \beta_j$ the degree of $D$. We call $g$ a {\it cone spherical metric representing $D$} if $K_g \equiv +1$ outside ${\rm supp} \, D$. We also note that the PDEs satisfied by cone spherical metrics form a special class of mean field equations, which are relevant to both Onsager's vortex model in statistical physics (\cite{CLMP92}) and the Chern-Simons-Higgs equation in superconductivity (\cite{CLW2004}). Similarly, we could define {\it cone hyperbolic metrics} or {\it cone flat ones} if their Gaussian curvatures equal identically $-1$ or $0$ outside the conical singularities. People naturally came up with

\begin{qu}
\label{qu:hfs}
Characterize all real divisors with coefficients in $(-1,\, \infty)\setminus\{0\}$ on $X$ which could be represented by cone hyperbolic, flat or spherical
metrics, respectively.
\end{qu}

The Gauss-Bonnet formula gives for Question \ref{qu:hfs} a natural necessary condition of
\[ {\rm sgn} \big (2 - 2g_X + \deg \, D \big ) = {\rm sgn}(K_g). \]
It is also sufficient for the cases of hyperbolic and flat metrics, and the hyperbolic or flat metric representing $D$ on $X$ exists uniquely  (\cite{Hei62, Tro86, McOwen88, Troyanov91}). The history of the research works of cone hyperbolic metrics goes back to {\' E}. Picard \cite{Pi1905} and H. Poincar{\' e} \cite{Po1898}. However, this natural necessary condition of $\deg \, D>2g_X-2$ is {\it not} sufficient for the existence of  cone spherical metrics (\cite{Tro89}). In this case Question \ref{qu:hfs} has been open over 20 years although many mathematicians had attacked or have been investigating it by using various methods and obtained a good understanding of the question (\cite{Troyanov91, UY2000, Er04, BdMM11, EGT1405, EG15, EGT1409, EGT1504, Er1706, CLW14, MP1505, MP1807, CWWX2015, SCLX2017, MZ1710}). Then we list some of the known results which are relevant to this manuscript. Troyanov proved a general existence theorem (\cite[Theorem 4]{Troyanov91}) on the problem of prescribing the Gaussian curvature on surfaces with conical singularities in subcritical regimes. It implies that there exists at least one cone spherical metric representing the $\mathbb{R}$-divisor $D=\sum_{j=1}^n\,\beta_j p_j$ with $-1<\beta_j\not=0$ on $X$ if
\[ 0 < 2 - 2g_X + \deg \, D < \min \, \Big(2, \, 2 + 2 \min_{1 \leq  j \leq n} \, \beta_j \Big). \]
Bartolucci-De Marchis-Malchiodi proved a general existence theorem (\cite[Theorem 1.1]{BdMM11}) on the same problem in supercritical regimes. In particular, they showed that there exists a cone spherical metric representing the effective $\mathbb{Z}$-divisor $D = \sum_{j = 1}^n \, \beta_j p_j$ on $X$ if the following conditions hold:
\begin{itemize}
\item $g_X>0$,
\item $\beta_j>0$ for all $1\leq j\leq n$,
\item $2-2g_X+\deg\, D>2$ and
\[ 2 - 2g_X + \deg D \notin \left \{ \mu > 0 \mid \mu = 2k + 2 \sum_{j=1}^n n_j (1 + \beta_j ), k \in \mathbb{Z}_{\geq 0}, n_j \in \{ 0,1\} \right \}. \]
\end{itemize}
Combining the results by Troyanov and Bartolucci-De Marchis-Malchiodi, we could see that if
$D = \sum_{j=1}^n\,\beta_j p_j$ is an effective $\mathbb{Z}$-divisor of {\it odd} degree on a compact Riemann surface $X$ of genus $g_X>0$, then there always exists a cone spherical metric representing $D=\sum_{j=1}^n\,\beta_j p_j$ provided the natural necessary condition of $\deg \, D>2g_X-2$ holds. The latter existence result
for cone spherical metrics was also obtained on elliptic curves independently by
Chen-Lin  as a corollary of a more general existence theorem \cite[Theorem 1.3]{CL15} for a class of mean field equations of Liouville type with singular data.

In this manuscript we would like to investigate cone spherical metrics with cone angles in $2\pi \mathbb{Z}_{>1}$, i.e. cone spherical metrics representing effective $\mathbb{Z}$-divisors. Roughly speaking, we shall establish an algebraic framework of such metrics and obtain a new existence theorem about these metrics
as an application of the framework.
In order to state them in detail, we need to prepare some notions.

We give a quick review of developing maps of cone spherical metrics representing an effective $\mathbb{Z}$-divisors and recall the concept of reducible/irreducible (cone spherical) metrics (\cite{UY2000, Er04, CWWX2015}). We call a non-constant multi-valued meromorphic function $f:X\to \mathbb{P}^1=\mathbb{C}\cup \{\infty\}$
{\it a projective function on $X$} if and only if the monodromy of $f$ lies in the group ${\rm PSL}(2,\,\mathbb{C})$ consisting of all M{\" o}bius transformations. Then, for a projective function $f$ on $X$, we could define its {\it ramification divisor} $R(f)$, which is an effective $\mathbb{Z}$-divisor on $X$. It was proved in \cite[Section 3]{CWWX2015} that there is a cone spherical metric representing an effective $\mathbb{Z}$-divisor $D$ on $X$ if and only if there exists a projective function $f$ on $X$ such that $R(f)=D$ and the monodromy of $f$ lies in
\[{\rm PSU(2)}:=\left\{z\mapsto \frac{az+b}{-\overline{b}z+\overline{a}}:|a|^2+|b|^2=1\right\}
\subset {\rm PSL}(2,\,\mathbb{C})\]
(we call that {\it $f$ has unitary monodromy} for short later on), and $g$ equals the pull-back $f^*g_{\rm st}$
of the stardard conformal metric $g_{\rm st}=\frac{4|dw|^2}{(1+|w|^2)^2}$ by $f$. At this moment, we call $f$ a {\it developing map} of the metric $g$, which is unique up to a pre-composition with a M{\" o}bius transformation in ${\rm PSU}(2)$. In particular,  it is well known that effective $\mathbb{Z}$-divisors represented by cone spherical metrics on the Riemann sphere $\mathbb{P}^1$ are exactly ramification divisors of rational functions on $\mathbb{P}^1$ (\cite[Theorem 1.9]{CWWX2015}), and hence all of them have even degree.
Recalling the universal (double) covering $\pi:{\rm SU(2)}\to {\rm PSU(2)}$, we make an observation (Corollary \ref{cor:criterion}) that an effective $\mathbb{Z}$-divisor $D$ represented by a cone spherical metric $g$ on $X$ has {\it even} degree if and only if the monodromy representation
$\rho_f:\,\pi_1(X)\to {\rm PSU(2)}$
of a developing map $f$ of the metric $g$ could be lifted to a group homomorphism $\widetilde{\rho_f}:\,\pi_1(X)\to {\rm SU}(2)$ such that there holds
the following commutative diagram
\[\xymatrix{
  \pi_{1}(X) \ar[r]^{\widetilde{\rho_f}} \ar[rd]_{\rho _{f}} & {\rm SU(2)} \ar[d] ^{\pi} \\
                                                                                & {\rm PSU(2)}
}\]

A cone spherical metric  is called {\it reducible} if and only if some developing map of it has monodromy in ${\rm U}(1)=\left\{z\mapsto e^{\sqrt{-1}\, t}z : t\in [0,\,2\pi)\right\}$. Otherwise, it is called {\it irreducible}. Q. Chen, W. Wang, Y. Wu and the last author \cite[Theorem 1.4-5]{CWWX2015} established a correspondence between meromorphic one-forms with simple poles and periods  in $\sqrt{-1}\mathbb{R}$ and general reducible cone spherical metrics, whose cone angles do not necessarily lie in $2\pi \mathbb{Z}_{>1}$.
In particular, an effective $\mathbb{Z}$-divisor $D$ represented by a reducible metric must have {\it even} degree since reducible metrics satisfy the lifting property in the last paragraph (\cite[Lemma 4.1]{CWWX2015}).
For simplicity, we may look at the degree of an effective $\mathbb{Z}$-divisor represented by a cone spherical metric {\it the degree of the metric}.
Recall the fact in the last paragraph that
if $D$ is an effective $\mathbb{Z}$-divisor with degree being odd and greater than $2g_X-2$ on a compact Riemann surface $X$ of genus $g_X>0$, there always exists a cone spherical metric representing $D$ on $X$. However,  the PDE method used in its proof seems {\it invalid} for the case of even degree. The reason lies in that there exists no a priori $C^0$ estimate for the corresponding PDE due to the blow-up phenomena caused by the one-parameter family of reducible metrics representing the same even degree effective $\mathbb{Z}$-divisor (\cite[Theorems 1.4-5]{CWWX2015}).
Therefore, not only does the following framework for irreducible metrics representing effective $\mathbb{Z}$-divisors shed new light on the connection between Differential Geometry and Algebraic Geometry underlying these metrics, but also it plays a crucial role for the existence problem of cone spherical metrics of even degree. We postpone the explanation of the relevant notions of various holomorphic bundles on compact Riemann surfaces until Section 2.

\begin{thm}
\label{thm:corr}
  Let $X$ be a compact connected Riemann surface of genus $g_X>0$. Then there exists a canonical correspondence
  between irreducible metrics representing effective $\mathbb{Z}$-divisors and line subbundles of rank two stable bundles on $X$. The detailed statement is divided into the following two cases according to the parity of the degree of metrics.
  \begin{enumerate}
    \item {\rm (Odd case)}
    We have the correspondence of \\

    $\left\{
      \begin{array}{c}
        \begin{tabular}{p{7em}}
          Cone spherical metric of odd degree
        \end{tabular}
      \end{array}
    \right\} \leftrightarrows
    \left\{
      \begin{array}{c}
        \begin{tabular}{p{14em}}
          Pair $(L,E)$ of a rank two stable bundle $E$ of degree $-1$ and a line subbundle $L$ of $E$
        \end{tabular}
      \end{array}
    \right\}.$

  We explain the details as follows.
  If $g$ is a cone spherical metric of odd degree and representing $D$ on $X$, then there exists a pair $(L,\,E)$ of a rank two stable vector bundle $E$ over $X$ and a line subbundle $L$ of $E$ such that $\deg\, E=-1$ and the section $s_{(L,\,E)}$ of the $\mathbb{P}^1$-bundle $\mathbb{P}(E)$ over $X$ defined by the embedding $L\hookrightarrow E$ forms a developing map of $g$. Moreover, we have
    \[\deg\, D=\deg\, E-2\deg\, L+2g_X-2.\]
   Conversely, each such pair $(L,\,E)$ defines an irreducible metric $g$ of odd degree in the sense $s_{(L,\,E)}$ defines a developing map of $g$.

    \item {\rm (Even case)} We have the correspondence of \\

    $\left\{
      \begin{array}{c}
        \begin{tabular}{p{8em}}
          Irreducible metric of even degree
        \end{tabular}
      \end{array}
    \right\} \leftrightarrows
    \left\{
      \begin{array}{c}
        \begin{tabular}{p{14em}}
          Pair $(L,E)$ of a rank two stable vector bundle $E$ with $\det\, E={\mathcal O}_X$ and a line subbundle $L$ of $E$
        \end{tabular}
      \end{array}
   \right\}.$

The details of the even case go similarly as the odd case. Furthermore, the irreducible metric corresponding to the pair $(L,\,E)$ represents an effective $\mathbb{Z}$-divisor lying  in the linear system $| K_X - 2L |$, where $K_X$ is the canonical line bundle of $X$.
\end{enumerate}
\end{thm}

\begin{rem}
{\rm
There exists also a correspondence between reducible metrics representing effective divisors and pairs
$(L,\,J\oplus J^*)$, where $J$ is a flat line bundle on $X$ and $L$ is a line subbundle of  $J\oplus J^*$ such that the section $s_{(L,J\oplus J^*)}$ of the
${\Bbb P}^1$-bundle ${\Bbb P}(J\oplus J^*)$ over X is {\it non-locally flat} (Definition \ref{def:nonflat}). Actually in Section 2 we encapsule it and Theorem \ref{thm:corr} into Theorem \ref{thm:correspondence}. 
These three correspondences are {\it not} one-to-one since tensoring a pair $(L,\,E)$ in any of them by an order-2 line bundle results in the same cone spherical metric. We will investigate this question carefully in a future paper.}
\end{rem}

We observe that a pair $(L,\,E)$ in the even case of Theorem \ref{thm:corr}  is nothing but a stable extension $E$ of $L^*$ by $L$, i.e.
\[0\to L\to E\to L^*\to 0\quad \text{with $E$ stable}.\]
Atiyah \cite{Atiyah57} proved that there exists no rank two stable vector bundle of even degree on  elliptic curves. Hence, cone spherical metrics of even degree are all reducible on elliptic curves. In order to find irreducible metrics of even degree,  we are naturally motivated to ask on a compact Riemann surface of genus greater than one the following question which is relevant to the Lange conjecture (\cite[p. 455]{Lan83}) solved by Russo-Teixidor i Bigas (\cite{RT99}) and Ballico-Russo (\cite{BR98, Ba2000}).

\begin{qu}
\label{qu:stable}
Let $L$ be a line bundle of negative degree on a compact Riemann surface $X$ of genus $g_X \geq 2$.
Does there exist a stable extension $E$ of $L^*$ by $L$?
\end{qu}

We note that Question \ref{qu:stable} in its particular setting is more refined than the Lange conjecture, which concerns the existence of stable extensions for two generic stable vector bundles which satisfy the natural slope inequality. We give an affirmative answer to the question which may be thought of as a Lange-type theorem.

\begin{thm}
\label{thm:stable}
Let $L$ be a line bundle of negative degree on a compact Riemann surface $X$ of genus $g_X\geq 2$.
There always exists a stable extension $E$ of $L^*$ by $L$.
\end{thm}

Then Theorems \ref{thm:corr} and \ref{thm:stable} give a new class of irreducible metrics of even degree.

\begin{cor}
\label{cor:irr}
Let $D$ be an effective $\mathbb{Z}$-divisor on a compact Riemann surface $X$ of genus $g_X\geq 2$ such that
$\deg\, D$ is even and greater than $2g_X-2$. Then there exists a cone spherical metric on $X$ representing some effective divisor linearly equivalent to $D$.
\end{cor}

\noindent {\bf Proof of the corollary} By the very condition of $L$, we could choose a negative line bundle $L$ such that $K_X-2L= \mathcal{O}_{X}(D)$. Then we choose a stable extension $E$ of $L^*$ by $L$ by Theorem \ref{thm:stable}. By Theorem \ref{thm:corr}, the pair $(L,\,E)$ corresponds to an irreducible metric representing some divisor in the complete linear system $|D|$. We note that by the Riemann-Roch theorem, the projective space of $|D|$ has dimension $(\deg\, D-g_X) \geq g_X\geq 2$. $\hfill{\Box}$

\begin{conj}
\label{conj:irr}
Under the assumption of Corollary \ref{cor:irr}, the effective divisors in $|D|$ represented by irreducible metrics on $X$ form a non-empty open subset of $|D|$ in the Zariski topology  of $|D|$.
\end{conj}

We speculate that there should also exist a parallel correspondence between general cone spherical metrics with
cone angles {\it not} necessarily lying in $2\pi{\Bbb Z}$ and parabolic line subbundles of rank two parabolic polystable bundles with parabolic degree {\it zero}. We shall investigate this correspondence in a future paper.
To conclude this introductory section, we explain briefly the organization of the left two sections of this manuscript. In Section 2, we shall establish a correspondence between cone spherical metrics representing effective divisors and non-locally flat sections of projective bundles associated to rank two polystable vector bundles, which contains Theorem \ref{thm:corr} as a particular case. In the last section, we use a result of Lange-Narasimhan (\cite[Corollary 1.2]{LN83}) and the induction argument to prove Theorem \ref{thm:stable}.

\section{Cone spherical metrics and indigenous bundles}
As we observed in Section 1, cone spherical metrics representing effective divisors  are equivalent to projective functions with unitary monodromy on compact Riemann surfaces, which naturally give {\it branched projective coverings} and {\it indigenous bundles} (see their definitions in Subsection \ref{subsec:branch}) on the Riemann surfaces. Moreover, such indigenous bundles are the associated projective bundles of rank two poly-stable bundles by the unitary monodromy property. In this way, we could establish the correspondence in Theorems \ref{thm:corr} and \ref{thm:correspondence}.
We state it  in Subsection \ref{subsec:state}, prepare the notions and a lemma of it in Subsection \ref{subsec:branch}, and prove it in Subsection \ref{subsec:pf}.

\subsection{Statement of the correspondence}
\label{subsec:state}

We need to prepare the notions of unitary flat holomorphic vector bundles and projective bundles on Riemann surfaces before stating the above-mentioned correspondence. Let $E$ be a holomorphic vector bundle of rank $r$ on $X$. We call that $E$ {\it has a unitary flat trivializations} if there exists a collection of trivializations $\psi_{\alpha}: E|_{U_{\alpha}} \rightarrow U_{\alpha} \times \mathbb{C}^r$ such that the corresponding transition functions
\[ g_{\alpha\beta}: U_{\alpha} \cap U_{\beta} \rightarrow U(r) \]
are all constants, where
\[ \psi_{\alpha} \circ \psi_{\beta}^{-1}(x, v) = \big(x, g_{\alpha\beta}(x)v \big) \quad \forall x \in U_{\alpha} \cap U_{\beta}, v\in \mathbb{C}^{r}. \]
Two such trivializations $\{U_{\alpha}, \psi_{\alpha}\}$ and $\{U_{\alpha}, \tilde{\psi}_{\alpha}\}$ are called equivalent if there exists a collection of maps $\varphi_{\alpha}: U_{\alpha} \times \mathbb{C}^r \rightarrow U_{\alpha} \times \mathbb{C}^r$ such that
\begin{itemize}
  \item $\varphi_{\alpha}(x, v)=\big( x, g_{\alpha}(x)v \big)$, where $g_{\alpha}: U_{\alpha} \rightarrow U(r)$ is a constant map.
  \item For all $\beta$, $\varphi_{\alpha} \circ \psi_{\alpha} \circ \psi_{\beta} ^{-1} = \tilde{\psi}_{\alpha} \circ \tilde{\psi}_{\beta} ^{-1}\circ \varphi_{\beta}$ in $(U_{\alpha}\cap U_{\beta}) \times \mathbb{C}^r$.
\end{itemize}
Then an equivalence class of such trivializations is called a {\it unitary flat structure} of $E$. A holomorphic vector bundle endowed with a unitary flat structure on it is called a {\it unitary flat vector bundle}. In other words, all the unitary flat vector bundles of rank $r$ on $X$ constitute the set $H^1(X, U(r))$. Similarly, we could define {\it projective unitary flat structures} of holomorphic projective bundles.

\begin{defn}
\label{def:nonflat} {\rm (\cite{Mand73})}
Suppose $P$ is a projective unitary flat $\mathbb{P}^1$-bundle on a Riemann surface $X$. Then we could choose a family of trivializations $\psi_{\alpha}: P|_{U_{\alpha}} \rightarrow U_{\alpha} \times \mathbb{P}^1$ such that the corresponding transition functions $g_{\alpha\beta}: U_{\alpha} \cap U_{\beta} \rightarrow PSU(2)$ are  constant maps. For any cross-section $s$ of $P$, $\psi_{\alpha} \circ s|_{U_{\alpha}}$ can be viewed as a map $s_{\alpha}: U_{\alpha} \rightarrow \mathbb{P}^1$. We call $s$ {\it non-locally flat} if $s_{\alpha}$ is a non-constant map for all $\alpha$. Since the transition functions of $P$ are constant, a cross-section $s$ is non-locally flat if and only if $s_{\alpha}$ is not a constant map for some $\alpha$. This property does not depend on the choice of the flat trivializations.
\end{defn}

Let $E$ be a unitary flat vector bundle of rank $2$ on $X$, and ${\frak s}$ a meromorphic section of $E$. Then the projective bundle $\mathbb{P}(E)$ is projective unitary flat. We call ${\frak s}$ {\it non-locally flat} if and only if its induced section $s$ of $\mathbb{P}(E)$ is non-locally flat.

Now we could state the correspondence between cone spherical metrics representing effective divisors and non-locally flat sections of rank two poly-stable vector bundles, which contains Theorem \ref{thm:corr} as a particular case.
\begin{thm}
\label{thm:correspondence}
  We are considering cone spherical metrics representing effective ${\Bbb Z}$-divisors on a compact Riemann surface $X$ of genus $g_X >0$. Then there exist the following three canonical correspondences
  \begin{enumerate}
    \item
    $\left\{
      \begin{array}{c}
        \begin{tabular}{p{8em}}
          Metric $g$ of odd degree
        \end{tabular}
      \end{array}
    \right\} \leftrightarrows
    \left\{
      \begin{array}{c}
        \begin{tabular}{p{14em}}
          The pair $(P,s)$, where $P = \mathbb{P}(E)$ for some stable vector bundle $E$ of rank $2$ with $\deg E = -1$ and $s$ is a non-locally flat section of $P$
        \end{tabular}
      \end{array}
    \right\}$

    \item
    $\left\{
      \begin{array}{c}
        \begin{tabular}{p{8em}}
          Irreducible  metric $g$ of even degree
        \end{tabular}
      \end{array}
    \right\} \leftrightarrows
    \left\{
      \begin{array}{c}
        \begin{tabular}{p{14em}}
          The pair $(E,\,{\frak s})$, where $E$ is a flat stable vector bundle of rank $2$ with trivial determinant bundle and ${\frak s}$ is a non-locally flat meromorphic section of $E$
        \end{tabular}
      \end{array}
    \right\}$

    \item
    $\left\{
      \begin{array}{c}
        \begin{tabular}{p{8em}}
          Reducible metric $g$
        \end{tabular}
      \end{array}
    \right\} \leftrightarrows
    \left\{
      \begin{array}{c}
        \begin{tabular}{p{14em}}
          The pair $(E,\,{\frak s})$, where $E=J \oplus J^*$ for some unitary flat line bundle $J$ and ${\frak s}$ is a non-locally flat meromorphic section of $E$
        \end{tabular}
      \end{array}
    \right\}$
  \end{enumerate}
The exact meaning of these three correspondences is the same as  Theorem \ref{thm:corr}.
\end{thm}

\subsection{Branched projective coverings}
\label{subsec:branch}

In this subsection, we at first introduce, among others, the notions of branched projective covering and indigenous bundle, which are crucial in the proof of Theorem \ref{thm:correspondence}. Then we prove Lemma \ref{lem:evendegree} which will be used in the proof of   Theorem
\ref{thm:correspondence}.

Let $X$ be a compact Riemann surface, and $\{U_{\alpha}, z_{\alpha}\}$ a holomorphic coordinate covering of $X$. If for each $\alpha$, $w_{\alpha}: U_{\alpha} \rightarrow \mathbb{P}^1$ is a non-constant holomorphic map such that $w_{\alpha} \circ w_{\beta}^{-1}(p) \in PSL(2, \mathbb{C})$ is independent of $p \in w_{\beta}(U_{\alpha} \cap U_{\beta})$, then $\{U_{\alpha}, w_{\alpha}\}$ is called a {\it branched projective covering} of $X$. Without loss of generality, we may assume that for each $\alpha$, $w_{\alpha}$ has at most one branch point $p_{\alpha}$ in $U_{\alpha}$ and $p_{\alpha} \not\in U_{\alpha} \cap U_{\beta}$ for all $\beta \neq \alpha$. For a branched projective covering $\{U_{\alpha}, w_{\alpha}\}$ of $X$, we call the effective divisor
\[ B_{\{U_{\alpha},w_{\alpha}\}} := \sum_{p \in X} \nu_{w_{\alpha}}(p)\cdot p\]
the {\it ramified divisor} of  $\{U_{\alpha}, w_{\alpha}\}$,
where $\nu_{w_{\alpha}}(p)$ is the branching order of $w_{\alpha}$ at the point $p$ \cite[Section 2]{Mand72}.
Then we could naturally associate a flat $\mathbb{P}^1$-bundle $P$ on $X$ (\cite[Section 2]{Gunning67}) to a branched projective covering  $\{U_{\alpha},w_{\alpha}\}$ of $X$, and obtain canonically
 a non-locally flat section defined by $$w_\alpha:U_\alpha\to {\Bbb P}^1\quad \text{for all}\quad \alpha$$ of $P$ (\cite[Section 2]{Gunning67}). The pair of such a flat $\mathbb{P}^1$-bundle and such a non-locally flat section is called an {\it indigenous bundle} associated to the branched projective covering $\{U_{\alpha}, w_{\alpha}\}$ on $X$ (\cite{Gunning67, Mand73}).

 \begin{lem}
 \label{lem:evendegree}
   Let $P$ be an indigenous bundle on $X$ associated to some branched projective covering $\{U_{\alpha}, w_{\alpha}\}$. Then the ramified divisor $B_{\{U_{\alpha},w_{\alpha}\}}$ has {\rm even} degree if and only if
   $P = \mathbb{P}(E)$ for some flat rank two vector bundle $E$ such that $\det(E) = \mathcal{O}_X$.
   Under this context,  there exists a meromorphic section ${\frak s} = \{({\frak s}_{1, \alpha}, {\frak s}_{2,\alpha})\}$ of $E$ such that $w_{\alpha} = \frac{{\frak s}_{1,\alpha}}{{\frak s}_{2,\alpha}}.$
   \end{lem}
\begin{proof} Since $\{U_{\alpha}, w_{\alpha}\}$ is a projective covering on $X$, we could choose
 a holomorphic coordinate covering  $\{U_{\alpha}, z_{\alpha}\}$  of $X$ and a family of matrices
$M_{\alpha\beta} = \begin{pmatrix}
   a_{\alpha\beta} & b_{\alpha\beta} \\
   c_{\alpha\beta} & d_{\alpha\beta}
\end{pmatrix} \in {\rm SL(2, \mathbb{C})}$ on all intersections $U_{\alpha} \cap U_{\beta}\not=\emptyset$ such that
\begin{align}
\label{wtransition}
  w_{\alpha} = \frac{a_{\alpha\beta}w_{\beta} + b_{\alpha\beta}}{c_{\alpha\beta}w_{\beta} + d_{\alpha\beta}}
\end{align}
and $w_\alpha=w_{\alpha}(z_{\alpha}):U_\alpha\to {\Bbb C}$ are holomorphic functions by using suitable M{\" o}bius transformations if necessary. Hence we have
$\frac{dw_{\alpha}}{dw_{\beta}} = \frac{1}{(c_{\alpha\beta}w_{\beta}+d_{\alpha\beta})^2}$.
Since $K_X$ is defined by the transition functions
$k_{\alpha\beta} = \frac{dz_{\beta}}{dz_{\alpha}}$,
we have
\[ \lambda_{\alpha\beta} := (c_{\alpha\beta}w_{\beta}+d_{\alpha\beta})^2 = \frac{w_{\beta}'(z_{\beta})} {w_{\alpha}'(z_{\alpha})} \cdot \frac{dz_{\beta}}{dz_{\alpha}} = h_{\alpha\beta}k_{\alpha\beta}, \quad \text{where}\quad
h_{\alpha\beta} := \frac{w_{\beta}'(z_{\beta})} {w_{\alpha}'(z_{\alpha})}.\]
Let $H$ be the line bundle defined by the transition functions $\{ h_{\alpha\beta} \}$. Then
$\{U_{\alpha}, \frac{1}{w_{\alpha}'(z_{\alpha})}\}$ forms a meromorphic section of $H$ so that
$H = \mathcal{O}_{X}\big(-B_{\{U_{\alpha}, w_{\alpha}\}}\big)$ has degree equal to
$\big(-\deg B_{\{U_{\alpha}, w_{\alpha}\}}\big)$
and the line bundle $\Lambda$ defined by the transition functions $\{ \lambda_{\alpha\beta} \}$ has degree of $\big(2g_{X} - 2 - \deg B_{\{U_{\alpha}, w_{\alpha}\}}\big)$.

Suppose that $\deg B_{\{U_{\alpha}, w_{\alpha}\}}$ is even. Then so is $\deg \Lambda$. Then there exists a line bundle $\xi$ defined by $\xi_{\alpha\beta}$ such that $\xi_{\alpha\beta}^2 = \lambda_{\alpha\beta}$. By changing the sign of $M_{\alpha\beta}\in {\rm SL(2,\,{\Bbb C})}$ if necessary,  we have $\xi_{\alpha\beta} = c_{\alpha\beta}w_{\beta} + d_{\alpha\beta}$. Taking a non-trivial  meromorphic section $\{U_{\alpha}, f_{\alpha}\}$ of $\xi$, we find
\[ \begin{pmatrix}
  w_{\alpha}f_{\alpha} \\
  f_{\alpha}
\end{pmatrix} = \begin{pmatrix}
  a_{\alpha\beta} & b_{\alpha\beta} \\
  c_{\alpha\beta} & d_{\alpha\beta}
\end{pmatrix} \begin{pmatrix}
  w_{\beta}f_{\beta} \\
  f_{\beta}
\end{pmatrix} \]
and $M_{\alpha\beta} = M_{\alpha\gamma} M_{\gamma\beta}$. The desired flat rank two vector bundle $E$ is just the one defined by $M_{\alpha\beta}$. Furthermore,  $(w_{\alpha}f_{\alpha}, f_{\alpha})$ forms a meromorphic section ${\frak s}=({\frak s}_{1,\alpha},\,{\frak s}_{2,\alpha})$ of $E$ satisfying $w_\alpha=\frac{{\frak s}_{1,\alpha}}{{\frak s}_{2,\alpha}}$.

Suppose that there exists a rank two flat vector bundle $E$ with ${\Bbb P}(E)=P$
for the indigenous bundle $P$ which is associated to the projective covering
$\{(U_\alpha,\,w_\alpha)\}$ and has the canonical section $s:=\{w_\alpha\}$.
By the argument in the first paragraph of \cite[Section 4]{Atiyah57II}, there exists a line subbundle
$L$ of $E$ generated by $s$ such that each non-trivial meromorphic section
${\frak s}=\big({\frak s}_{1,\alpha},\,{\frak s}_{2,\,\alpha}\big)$ of the line bundle $L\subset E$ satisfies
$w_\alpha=\frac{{\frak s}_{1,\alpha}}{{\frak s}_{2,\alpha}}$.
Denote by
$ M_{\alpha\beta} = \begin{pmatrix}
   a_{\alpha\beta} & b_{\alpha\beta} \\
   c_{\alpha\beta} & d_{\alpha\beta}
\end{pmatrix} \in {\rm SL(2, \mathbb{C})}$
the transition functions of $E$. Recalling \eqref{wtransition} in the first paragraph of the proof,
we only need to show the degree of the line bundle $\Lambda$ is even since $\Lambda = H \otimes K_X$ and
\[\deg\, B_{\{U_\alpha,\,w_\alpha\}}\equiv \deg\, H\quad (\text{mod}\ 2).\]
 Since
$w_\alpha={\frak s}_{1,\alpha}/{\frak s}_{2,\alpha}$, we have
\[ \begin{pmatrix}
  w_{\alpha}{\frak s}_{2,\alpha} \\
  {\frak s}_{2,\alpha}
\end{pmatrix} = \begin{pmatrix}
  {\frak s}_{1,\alpha} \\
  {\frak s}_{2,\alpha}
\end{pmatrix} = \begin{pmatrix}
  a_{\alpha\beta} & b_{\alpha\beta} \\
  c_{\alpha\beta} & d_{\alpha\beta}
\end{pmatrix} \begin{pmatrix}
 {\frak s}_{1,\beta} \\
  {\frak s}_{2,\beta}
\end{pmatrix} = \begin{pmatrix}
  a_{\alpha\beta} & b_{\alpha\beta} \\
  c_{\alpha\beta} & d_{\alpha\beta}
\end{pmatrix} \begin{pmatrix}
  w_{\beta}{\frak s}_{2,\beta} \\
  {\frak s}_{2,\beta}
\end{pmatrix} \]
and ${\frak s}_{2, \alpha} = (c_{\alpha\beta}w_{\beta} + d_{\alpha\beta}) {\frak s}_{2,\beta}$. It follows that the functions $\{c_{\alpha\beta}w_{\beta} + d_{\alpha\beta}\}$ on $U_{\alpha} \cap U_{\beta}$ are 1-cocyles, which define a line bundle $\xi$ with $\xi^2 = \Lambda$. Hence, $\deg \Lambda$ is even.
\end{proof}

Let $g$ be a cone spherical metric on $X$ representing an effective ${\Bbb Z}$-divisor $D = \sum_{i=1}^n \beta_i p_i$. Suppose that $\{U_{\alpha}, z_{\alpha}\}$ is a holomorphic coordinate covering of $X$ such that each
$U_\alpha$ is simply connected and contains at most one point in ${\rm Supp}\,D$. Then there exists a holomorphic map $f_{\alpha}: U_{\alpha} \rightarrow \mathbb{P}^1$ for all $\alpha$ such that it has at most one ramified point and
$g|_{U_{\alpha}} = f_{\alpha}^*\,g_{\rm st}$ (\cite[Lemmas 2.1 and 3.2]{CWWX2015}), where
$g_{\rm st}=\frac{4|dw|^2}{(1+|w|^2)^2}$ is the standard metric on ${\Bbb P}^1$. Then these pairs $\{U_{\alpha}, f_{\alpha}\}$ define a branched projective covering of $X$ with ramified divisor being $D$ such that
  \[ f_{\alpha\beta} = f_{\alpha} \circ f_{\beta}^{-1}(x) \in {\rm PSU(2)} \]
is  independent of $x \in f_{\beta}(U_{\alpha} \cap U_{\beta})$. Then we obtain an indigenous bundle $P$ with structure group $PSU(2)$ on $X$ associated to the branched projective covering $\{U_\alpha,\, f_{\alpha}\}$ and a canonical section $s=\{f_{\alpha}\}$ of $P$ which is non-locally flat.

As a small application of Lemma \ref{lem:evendegree}, we could give a criterion for the parity of the degree of effective $\mathbb{Z}$-divisors represented by cone spherical metrics as the following

\begin{cor}
\label{cor:criterion}
Let $D$ be an effective $\mathbb{Z}$-divisor represented by a cone spherical metric $g$ on $X$. Then $\deg D$ is even if and only if the monodromy representation $\rho_f:\,\pi_1(X)\to {\rm PSU(2)}$ of a developing map $f$ of $g$ could be lifted to $\widetilde{\rho_f}:\,\pi_1(X)\to {\rm SU}(2)$
such that there holds
the following commutative diagram
\[\xymatrix{
  \pi_{1}(X) \ar[r]^{\widetilde{\rho_f}} \ar[rd]_{\rho _{f}} & {\rm SU(2)} \ar[d] ^{\pi} \\
                                                                                & {\rm PSU(2)}
}
\]
\end{cor}

\subsection{Proof of Theorems \ref{thm:correspondence} and \ref{thm:corr}}
\label{subsec:pf}

With the help of Lemma \ref{lem:evendegree}, we can now prove Theorem \ref{thm:correspondence}.
\begin{proof}[Proof of Theorem \ref{thm:correspondence}]
The proof is divided into three parts according to the three correspondences.

\begin{itemize}
\item[\textbf{Part I}]
  In this part, we will prove the first correspondence.

  Suppose at first that $g$ is a cone spherical metric on $X$ representing an effective ${\Bbb Z}$-divisor $D = \sum_{i=1}^n \beta_i p_i$ of odd degree. Then, by the argument before Corollary \ref{cor:criterion}, we could obtain a branched projective covering $\{U_{\alpha}, f_{\alpha}\}$ with ramified divisor $D$ such that $g|_{U_\alpha}=f^*_\alpha\, g_{\rm st}$ and an indigenous bundle $P$ associated to this covering such that the transition functions $\{U_{\alpha}, f_{\alpha\beta}\}$ lies in ${\rm PSU(2)}$. Moreover, $\{U_{\alpha}, f_{\alpha}\}$ defines a non-locally flat section $s$ of $P$.
  Since $X$ is a compact Riemann surface, there exists a rank two holomorphic vector bundle $E$ on $X$ such that $P = \mathbb{P}(E)$. Since $\deg\, D$ is odd, by Corollary \ref{cor:criterion}, the monodromy representation $\rho:\pi_1(X)\to {\rm PSU}(2)$ of the flat bundle $P$, which coincides with the one of the metric $g$, could not be lifted to ${\rm SU(2)}$. By Lemma \ref{lem:evendegree},  $\deg\, E$ is odd. In particular, the metric $g$ is irreducible. Hence,  we could assume $P = \mathbb{P}(E)$ with $\deg E = -1$ without loss of generality.

  Next, let $(P, s)$ be a pair on the right-hand side of the correspondence.
  Since $E$ is a stable vector bundle on $X$, $P = \mathbb{P}(E)$ comes from an irreducible representation ${\rho}: \pi_1(X) \rightarrow {\rm PSU(2)}$ (\cite{NS1965}). Let $\{U_{\alpha}, s_{\alpha}\}$ be the local expression of the non-locally flat section $s$. On all intersections $U_{\alpha} \cap U_{\beta}$, the local sections satisfy $ s_{\alpha} = \phi_{\alpha\beta} \circ s_{\beta}$,
  where the transition functions $\phi_{\alpha\beta}$ lies in $PSU(2)$. Defining $g_{\alpha} = s_{\alpha}^*g_{\rm st}$ on each $U_{\alpha}$, we find $g_{\alpha}|_{U_{\alpha}\cap U_{\beta}} = g_{\beta}|_{U_{\alpha}\cap U_{\beta}}$ since $\phi_{\alpha\beta} \in {\rm PSU(2)}$. Hence, the metrics $g_{\alpha}$ on $U_\alpha$'s define a global cone spherical metric $g = s^*g_{st}$ representing an effective ${\Bbb Z}$-divisor $D$ on $X$.
  If the degree of $D$ is even, then $\mathbb{P}(E) = P = \mathbb{P}(E')$ for some flat vector bundle $E'$ by Lemma \ref{lem:evendegree}. Hence $\deg E$ is even. Contradiction!


\item[\textbf{Part II}] In this part, we will prove the second correspondence.
    Let $g$ be an irreducible metric on $X$ of even degree. By the discussion of the second paragraph of Part I, we know that $P$ is an indigenous bundle and $\deg B_{\{U_{\alpha}, f_{\alpha}\}}$ is an even number. Then by Lemma \ref{lem:evendegree}, we could find the pair $(E,\,{\frak s})$ as desired. On the other hand,
 given such a pair $(E,\,{\frak s})$, let $\{U_{\alpha}, (s_{1,\alpha}, s_{2, \alpha})\}$ be a local expression of ${\frak s}$, which defines a branched projective covering of $X$. Then $\{U_{\alpha}, \frac{{\frak s}_{1,\alpha}}{{\frak s}_{2,\alpha}}\}$ defines a non-locally flat section of the associated indigenous bundle $\mathbb{P}(E)$. By the same argument in the third paragraph of Part I, we could define an irreducible cone spherical metric representing an effective ${\Bbb Z}$-divisor $D$ on $X$.  Since $D$ is the ramified divisor of the branched projective covering, it has even degree by Lemma \ref{lem:evendegree}.

\item[\textbf{Part III}]  In this part, we will prove the last correspondence.

  As in Part I, let $\{U_{\alpha}, f_{\alpha}\}$ be a branched projective covering of $X$ induced by a reducible metric $g$, to which the indigenous $P$ is associated.  Hence, the monodromy representation $\rho: \pi_1(X) \rightarrow {\rm PSU(2)}$ of the flat bundle $P$ can be lifted to $\tilde{\rho}: \pi_1(X) \rightarrow SU(2)$ since $g$ is reducible. In particular, the metric $g$ is of even degree.
   Since $g$ is reducible, by Lemma 4.1 in \cite{CWWX2015}, we have,  up to a conjugation,
\[ \tilde{\rho}(\pi_1(X)) \in \{\text{diag}(e^{i\theta}, e^{-i\theta}): \theta \in \mathbb{R}\} \subset {\rm SU(2)}. \]
  It follows that there exist two flat line bundles $J_1$ and $J_2$ on $X$ such that $E = J_1 \oplus J_2$ and
  $J_1 = J_2^*$.

  Now suppose $E = J \oplus J^*$ for some flat line bundle $J$ and ${\frak s}$ is a non-locally flat meromorphic section of $E$. Since $X$ is a compact Riemann surface, we could choose transition functions $j_{\alpha\beta}$ of $J$ lying in ${\rm U(1)}=\{z\in {\Bbb C}:|z|=1\}$ under a suitable trivialization $\{U_\alpha\}$ of $J$. Then we could write ${\frak s} = ({\frak s}_{1,\alpha}, \, {\frak s}_{2,\alpha})$ with
  \[ {\frak s}_{1,\alpha} = j_{\alpha\beta} \cdot {\frak s}_{1,\beta} \qquad {\frak s}_{2,\alpha} = j_{\alpha\beta}^{-1} \cdot {\frak s}_{2,\beta} \]
in the intersections $U_{\alpha} \cap U_{\beta}$. Then $\{U_{\alpha}, u_{\alpha} := \frac{s_{1, \alpha}} {s_{2, \alpha}}\}$ defines a holomorphic section of $\mathbb{P}(E)$ and $u_{\alpha} = j_{\alpha\beta}^2 \cdot u_{\beta}$. Denoting by $g_{\alpha}$ the metric on $U_{\alpha}$ defined by $u_{\alpha}^*(g_{\rm st})$, we find that these metrics coincide with each other on all intersections $U_{\alpha} \cap U_{\beta}$ since $j_{\alpha\beta}^2 \in {\rm U(1)}$. Hence we obtain a globally defined reducible cone spherical metric $g$.

\end{itemize}

In summary, we  complete the proof of Theorem \ref{thm:correspondence}.

\end{proof}

In order to complete the proof of Theorem \ref{thm:corr}, we need the following two lemmas.

\begin{lem}
\label{lem:doublecover}
Let $X$ be a compact Riemann surface of genus $g_{X} > 0$. Then there exists a compact connected Riemann surface $\tilde{X}$ of genus $2g_{X} - 1$ with an unramified double cover $\pi: \tilde{X} \rightarrow X$.
\end{lem}
\begin{proof}
Since $g_X>0$, we could choose a non-trivial line bundle $L$ such that $L\otimes L={\mathcal O}_X$. Defining
\[ \tilde{X} = \Big\{ \big(x, \tau(x)\big)\in L \mid x \in X, \tau(x) \in L_{x} \text{ such that } \tau(x) \otimes \tau(x) = 1\Big\}, \]
as Exercise 1 in \cite[Chapter 2]{Voisin02}, we know that the natural projection $\pi: \tilde{X} \rightarrow X$ is an unramified double cover. Suppose that $\tilde{X}$ is not connected. Then $\tilde{X}$ must be a disjoint union of two copies of $X$, i.e.
$\tilde{X} = X_{1} \sqcup X_{2}$,
where $X_{1} \cong X \cong X_{2}$. The isomorphism $X \rightarrow X_{1} \subset \tilde{X}$ gives a nowhere vanishing section $\sigma$ of $L$, which contradicts that $L$ is non-trivial.
\end{proof}

By using  Lemma \ref{lem:evendegree} and Lemma \ref{lem:doublecover}, we can prove
\begin{lem}
\label{lem:degreerelation}
  Let $\{U_{\alpha}, w_{\alpha}\}$ be a branched projective covering of $X$ of genus $g_{X} >0$ and $P = \mathbb{P}(E)$ the indigenous bundle associated to it, where $E$ is a rank two holomorphic vector bundle on $X$. Suppose that $L$ is the line subbundle of $E$ defined by the canonical section $\{U_{\alpha}, w_{\alpha}\}$ of the indigenous bundle $\mathbb{P}(E)$. Then
  \begin{equation}
  \label{equ:rel}
   \deg\,\big(B_{\{U_{\alpha}, w_{\alpha}\}}\big)=\deg(E)-2\deg(L)+2g_{X}-2. \end{equation}
\end{lem}
\begin{proof} At first we note that
  the right hand side of \eqref{equ:rel} does not depend on the choice of $E$.

  Suppose that $\deg B_{\{U_{\alpha}, w_{\alpha}\}}$ is even.
  We could assume that $\det\, E={\mathcal O}_X$.
   Recalling the second paragraph in the proof of Lemma \ref{lem:evendegree}, we find that ${\frak s}=(w_{\alpha}f_{\alpha}, f_{\alpha})^t$ is a meromorphic section of $E$ and the divisor ${\rm div}(s)$(see the definition in the first paragraph of \cite[Section 4]{Atiyah57II}) associated to $s$ coincides with ${\rm div}(f_\alpha)$. Both the section ${\frak s}$ of $E$ and the section $\{U_{\alpha}, w_{\alpha}\}$ of ${\Bbb P}(E)$ define the same line  subbundle $L$ of $E$, which coincides with $\xi$.
Hence  we have  $$\mathcal{O}_{X}(-B_{\{U_{\alpha}, w_{\alpha}\}}) \otimes K = H \otimes K = \Lambda = \mathcal{O}_{X}(-2\,\text{div}(s))=L^2.$$
We are done for this case since $\deg\, E=0$. Actually, we obtain more than \eqref{equ:rel} that the ramified divisor $B_{\{U_{\alpha}, w_{\alpha}\}}$ of the branched projective covering  is linear equivalent to $K_X-2L$.

  Suppose that $\deg B_{\{U_{\alpha}, w_{\alpha}\}}$ is odd. Since $g_X>0$, there exists an unramified double cover $\pi: \tilde{X} \rightarrow X$ such that $\tilde{X}$ is connected. Without loss of generality, we may assume that $U_{\alpha}$ is sufficient small such that $\left\{\pi^{-1}(U_{\alpha}), \pi^*(w_{\alpha})\right\}$ is a branched projective covering of $\tilde{X}$ and $\pi^*(P) = \mathbb{P}(\pi^*(E))$ is the indigenous bundle associated to $\{\pi^{-1}(U_{\alpha}), \pi^*(w_{\alpha})\}$. Then we have the following three equalities
  \begin{eqnarray*}
   \deg B_{\{\pi^{-1}(U_{\alpha}), \pi^*(w_{\alpha})\}} &=&
   2\deg B_{\{U_{\alpha}, w_{\alpha}\}}, \\ \deg \pi^*(L) &=& 2 \deg L,\\ \deg\pi^*(E) &=& 2 \deg E. \end{eqnarray*}
  By using \eqref{equ:rel} on $\tilde{X}$, we have
  \[ \deg \pi^*(L) = g_{\tilde{X}}-1 + (\deg \pi^*(E) - \deg B_{\{\pi^{-1}(U_{\alpha}), \pi^*(w_{\alpha})\}})/2 \]
  and
  \[ \deg (L) = g_X-1 +(\deg(E) -\deg(B_{\{U_{\alpha}, w_{\alpha}\}}))/2. \]
\end{proof}

\begin{proof}[Proof of Theorem \ref{thm:corr}]

Let $E$ be a rank two stable vector bundle on $X$ with $\deg(E) = -1$ or $0$. Then the projective bundle $\mathbb{P}(E)$ is a unitary flat $\mathbb{P}^1$-bundle. The line subbundle of $E$ defined by a section of ${\Bbb P}(E)$ is non-trivial since $E$ is stable and has degree $-1$ or $0$. On the other hand,  for any given locally flat section $\phi$ of $E$, by the construction in the first paragraph of \cite[Section 4]{Atiyah57II}, we can see that the divisor $\text{div}(\phi)$ associated to $\phi$ vanishes and the line subbundle defined by $\phi$ is trivial. Therefore, each section of $\mathbb{P}(E)$ is non-locally flat. By combining with Theorem \ref{thm:correspondence}, we get the two correspondences in Theorem \ref{thm:corr}.

The equality in the first correspondence follows from Lemma \ref{lem:degreerelation}.

Let $D$ be the effective ${\Bbb Z}$-divisor represented by the metric given by the pair $(E,\,L)$ in the second correspondence. Recalling the equality in the second paragraph of the proof of Lemma \ref{lem:degreerelation}, we find that $D$ lies in the complete linear system $|K_X-2L|$.

In summary, we complete the proof of Theorem \ref{thm:corr}.
\end{proof}

\section{A Lange-type theorem}
\begin{proof} [Proof of Theorem \ref{thm:stable}]
We shall use the inductive argument  to prove the theorem.

If $\deg L = -1$, then there exists an extension of $L^{*}$ by $L$
\[ 0 \rightarrow L \rightarrow E \rightarrow L^{*} \rightarrow 0 \]
such that $\deg L' \leq \deg L$ for all line subbundles $L' \subset E$ \cite[Corollary 1.2]{LN83}. Hence $E$ is stable and has trivial determinant bundle.

Next suppose that for all line bundles $L$ with $-d < \deg L < 0$, there exists a rank two stable vector bundle $E$ such that $L$ is a line subbundle of $E$ and $\det E = \mathcal{O}_{X}$.

Now we consider a line bundle $L$ of degree $-d \leq -2$. Choose a point $p \in X$. Then there exists a rank two stable vector bundle $E$ with $\det E = \mathcal{O}_{X}$  such that $L(p):= L \otimes \mathcal{O}_{X}(p)$ is a line subbundle of $E$. That is, we have the following short exact sequence of locally free sheaves
\begin{equation}
\label{equ:split}
 0 \rightarrow L(p) \rightarrow E \rightarrow E/L(p) \rightarrow 0.
\end{equation}
Hence $L$ is a line subbundle of $E(-p):= E \otimes \mathcal{O}_{X}(-p)$. Consider all the sheaves $\mathcal{F}$ which fit into
\begin{equation}
\label{equ:keyexact}
 E(-p) \subsetneq \mathcal{F} \subsetneq E.
\end{equation}
Since all the sheaves $\mathcal{F}$ on the algebraic curve $X$ have no torsion, they are locally free and are parametrized by 1-dimensional linear subspaces $\mathcal{F}/ E(-p)$ of $E / E(-p) \cong \mathbb{C}^{2}_{p}$.



We write the stalk $E_{p}$ of $E$ at $p$ as
\[ E_{p} = \mathcal{O}_{X,p} \langle (1,0), (0,1) \rangle = \{(s_{1}, s_{2}) \mid s_{1}, s_{2} \in \mathcal{O}_{X,p}\}. \]
Choose a local complex coordinate $z$ centered at $p$. Then we could express $E(-p)_p$ by
\[ E(-p)_{p} = \mathcal{O}_{X,p} \langle (z,0), (0,z) \rangle. \]
Since $\mathcal{F}/ E(-p)$ is a subspace of $E / E(-p)$, by changing the basis of $E_{p}$ if necessary, we could assume that
\[ \mathcal{F}_{p} = \mathcal{O}_{X,p} \langle (z,0), (0,1) \rangle. \]
Therefore
\[ \mathcal{F}_{p}/ E(-p)_{p} = \{(0,v) \mid v \in \mathbb{C} \}. \]
By the exactness of \eqref{equ:split}, we have $E_{p} = L(p)_{p} \oplus (E/L(p))_{p}$ as $\mathcal{O}_{X,p}$-modules. Hence $L(p)_{p}$ can be expressed as
\[ L(p)_{p} = \mathcal{O}_{X,p}\langle (A_{1}, A_{2}) \rangle, \]
where $A_{1}, A_{2} \in \mathcal{O}_{X,p}$ and there exist $B_{1}, B_{2} \in \mathcal{O}_{X,p}$ such that
\[\begin{pmatrix}
   A_{1} & B_{1} \\
   A_{2} & B_{2}
\end{pmatrix} \in GL(2, \mathcal{O}_{X,p}). \]
Moreover
\[ L_{p} = \mathcal{O}_{X,p}\langle (zA_{1}, zA_{2}) \rangle. \]

On the other hand, the subspace of $\mathbb{C}^{2}$ corresponding to the embedding of fiber $L(p)|_{p} \rightarrowtail E|_{p}$ is generated by the vector $\big(A_{1}(p), A_{2}(p)\big)$. Hence the 1-dimensional subspace of $\mathbb{C}_{p}^{2}$ corresponding to $\mathcal{F}/E(-p)$ coincides with the subspace as $L(p)|_{p} \rightarrowtail E|_{p} = \mathbb{C}^{2}$ if and only if $A_{1} \not\in \mathcal{O}_{X,p}^{*}$.

If $A_{1} \not\in \mathcal{O}_{X,p}^{*}$, then we could assume $A_{1} = zA_{1}', A_{1}' \in \mathcal{O}_{X,p}$. Hence
\[ \big(zA_{1}, zA_{2}\big) = \big(z^{2}A_{1}', zA_{2}\big) = \Big( (z,0), (0,1)\Big) \begin{pmatrix}
   zA_{1}' \\
   zA_{2}
\end{pmatrix}.\]
Therefore $L_{p}$ is not a direct summand in $\mathcal{F}_{p}$, which means $\mathcal{F}_{p} / L_{p}$ is not a locally free $\mathcal{O}_{X,p}$-module. Similarly, we obtain that  if $A_{1} \in \mathcal{O}_{X,p}^{*}$, then  $A_{1}(p) \neq 0$ and $\mathcal{F}_{p} / L_{p}$ is a locally free $\mathcal{O}_{X,p}$-module.

In summary, {\it $\mathcal{F}/L$ is not locally free if and only if the 1-dimensional subspace of $\mathbb{C}_{p}^{2}$ corresponding to $\mathcal{F}$ coincides with the subspace as $L(p)|_{p} \rightarrow E|_{p}$.} Now we fix a locally free sheaf $\mathcal{F}$ such that $\mathcal{F}/L$ is locally free. Moreover $\det \mathcal{F} = \mathcal{O}_{X}(-p)$ by \eqref{equ:keyexact}.

Similarly, consider all the locally free sheaves of degree zero fit into
\[  \mathcal{F} \subsetneq \mathcal{G} \subsetneq \mathcal{F}(p). \]
Then $\det \mathcal{G} = \mathcal{O}_{X}$ and $\mathcal{G}/L$ is locally free for $\mathcal{G}$ being generic. We only need to show that a generic $\mathcal{G}$ is stable. Actually,
by \eqref{equ:keyexact}, we have
\[ E(-p) \subsetneq \mathcal{F} \subsetneq E \subsetneq \mathcal{F}(p). \]
Since $\mathcal{G} = E$ is stable and all the sheaves $\mathcal{G}$ are parametrized by $\mathbb{P}^{1}$, we are done by  the openness of stability \cite[Theorem 2]{NS1965}.
\end{proof}

\begin{center}
 {\bf Acknowledgements}
\end{center}

\noindent Song and Xu would like to thank Professor Botong Wang at University of Wisconsin Madison for his generous help with the proof of Theorem 1.2 and a long stimulating conversation on parabolic bundles. The authors would like to express their great gratitude to Professors Kang Zuo, Xiaotao Sun and Mao Sheng for their constant interest and encouragement during the course of this work.
Li is supported in part by the National Natural Science Foundation of China (Grant No. 11501418 ) and Shanghai Sailing Program (15YF1412500).
Song is partially supported by National Natural Science Foundation of China (Grant Nos. 11471298, 11622109 and 11721101) and Tianjin Outstanding Talent Fund.
Xu is supported in part by the National Natural Science Foundation of China (Grant No. 11571330).
Both Song and Xu are supported in part by the Fundamental Research Funds for the Central Universities.

\vspace{0.5cm}

{\sc \noindent Lingguang Li\\
School of Mathematical Sciences\\
Tongji University\\
Shanghai 200092 China}\\
LiLg@tongji.edu.cn\\

{\small {\sc  \noindent Jijian Song\\
Wu Wen-Tsun Key Laboratory of Math, USTC, CAS\\
School of Mathematical Sciences\\
University of Science and Technology of China\\
Hefei 230026 China

\noindent Center for Applied Mathematics\\
School of Mathematics, Tianjin University\\
Tianjin 300350 China}\\
smath@mail.ustc.edu.cn\\

{\sc \noindent Bin Xu\\
Wu Wen-Tsun Key Laboratory of Math, USTC, CAS\\
School of Mathematical Sciences\\
University of Science and Technology of China\\
Hefei 230026 China}\\
\Envelope bxu@ustc.edu.cn}


\begin{thebibliography}{99}

\bibitem{Atiyah57} M. F. Atiyah, \emph{Complex analytic connections in fibre bundles}, Trans. Amer. Math. Soc. {\bf 84} (1957) 181-207.

  \bibitem{Atiyah57II} M. F. Atiyah, \emph{Vector bundles over an elliptic curve}, Proc. Lond. Math. Soc. {\bf 84}:3 (1957) 414-452.

 \bibitem{Ba2000} E. Ballico, \emph{Extensions of stable vector bundles on smooth curves: Lange's conjecture}, An. Stiint. Univ. Al. I. Cuza Iasi. Mat. (N.S.) {\bf 46}:1 (2000), 149-156.

 \bibitem{BR98} E. Ballico and B. Russo, \emph{Exact sequence of stable bundles on projective curves},
 Math. Nachr. {\bf 194}:1 (1998) 5-11.

   \bibitem{BdMM11} D. Bartolucci, F. De Marchis and A. Malchiodi, \emph{Supercritical conformal metrics on surfaces with conical singularities}, IMRM, Vol. 2011, No. 24, 5625-5643.


  \bibitem{CLMP92} E. Caglioti, P. L. Lions, C. Marchioro and M. Pulvirenti, \emph{A special class of stationary flows for two-dimensional Euler equations: a statistical mechanics description}, Comm. Math. Phys. {\bf 143} (1992) 501-525.

  \bibitem{CLW14} C.-L. Chai, C.-S. Lin and C.-L. Wang, \emph{Mean field equations, hyperelliptic curves and modular forms: I}, Cambridge J. Math. {\bf 3}(2015) 127-274.

  \bibitem{CL15} C.-C. Chen and C.-S. Lin, Mean field equation of Liouville type with singular data: topological degree, Comm. Pure Appl. Math.,  {\bf 68}:6 (2015) 887-947.

  \bibitem{CLW2004} C.-C. Chen, C.-S. Lin and G. Wang, \emph{Concentration phenomena of two-vertex solutions in a Chern-Simons model}, Ann. Scuola Norm. Sup. Pisa Cl. Sci. {\bf 3}:5 (2004) 367-397.

  \bibitem{CWWX2015} Q. Chen, W. Wang, Y. Wu and B. Xu, \emph{Conformal metrics with constant curvature one and finitely many conical singularities on compact Riemann surfaces}, Pacific J. Math. {\bf 273}:1 (2015) 75-100.

  \bibitem{Er04} A. Eremenko, \emph{Metrics of positive curvature with conic singularities on the sphere}, Proc. AMS, 132 (2004) 3349-3355.

  \bibitem{Er1706} A. Eremenko, \emph{Co-axial monodromy}, arXiv:1706.04608v1

  \bibitem{EGT1405} A. Eremenko,  A. Gabrielov and V. Tasarov, \emph{Metrics with conic singularities and spherical
polygons},  Illinois J. Math. {\bf 58}:3 (2014) 739-755. 

  \bibitem{EG15} A. Eremenko and A. Gabrielov, \emph{On metrics of curvature 1 with four conic singularities on tori and on the sphere}, Illinois J. Math. {\bf 59}:4 (2015) 925-947.

  \bibitem{EGT1409} A. Eremenko,  A. Gabrielov and V. Tasarov, \emph{Metrics with four conic singularities and
spherical quadrilaterals}, Conform. Geom. Dyn. {\bf 20} (2016) 128-175. 

  \bibitem{EGT1504} A. Eremenko,  A. Gabrielov and V. Tasarov, \emph{Spherical quadrilaterals with three non-integer
angles}, Zh. Mat. Fiz. Anal. Geom., {\bf 12}:2 (2016) 134-167



  \bibitem{Gunning67} R. C. Gunning, \emph{Special coordinate coverings of Riemann surfaces}, Math. Ann. {\bf 170}(1967), 67-86.



  \bibitem{Hei62} M. Heins, \emph{On a class of conformal metrics.} Nagoya Math. J. {\bf 21} (1962) 1-60.


  \bibitem{Lan83} H. Lange,  \emph{Zur Klassifikation von Regelmanningfaltigkeiten}, Math. Ann. {\bf 262}
(1983), 447-459.

  \bibitem{LN83} H. Lange and M. S. Narasimhan, \emph{Maximal sub bundles of rank two bundles on curves}, Math. Ann. {\bf 266} (1983) 55-72.


  \bibitem{Mand72} R. Mandelbaum, \emph{Branched structures on Riemann surfaces}, Trans. Amer. Math. Soc. {\bf 163} (1972) 261-275.

  \bibitem{Mand73} R. Mandelbaum, \emph{Branched structures and affine and projective bundles on Riemann surfaces}, Trans. Amer. Math. Soc. {\bf 183} (1973) 37-58.

  \bibitem{MZ1710} R. Mazzeo and X. Zhu, \emph{Conical metrics on Riemann surfaces, I: the
compactified configuration space and regularity}, arXiv:1710.09781v1

  \bibitem{McOwen88} R. C. McOwen, \emph{Point singularities and conformal metrics on Riemann surfaces}, Proc. Amer. Math. Soc. {\bf 103}:1 (1988) 222-224.

  \bibitem{MP1505} G. Mondello and D. Panov, \emph{Spherical metrics with conical singularities on a 2-sphere: angle constraints}, Int. Math. Res. Not. IMRN 2016, no. 16, 4937-4995.

   \bibitem{MP1807} G. Mondello and D. Panov, \emph{Spherical surfaces with conical points: systole inequality and moduli spaces with many connected components}, arXiv:1807.04373v1

  \bibitem{NS1965} M. S. Narasimhan and C. S. Seshadri, \emph{Stable and unitary vector bundles on a compact Riemann surface}, Ann. of Math. {\bf 82} (1965) 540-567.

  \bibitem{Pi1905} {\' E}. Picard, \emph{De l'int{\' e}gration de l'{\' e}quation $\Delta u =e^u$ sur une surface
Riemann ferm{\' e}e}, J. reine angew. Math. {\bf 130} (1905) 243-258.

\bibitem{Po1898} H. Poincar{\' e}, \emph{Fonctions fuchsiennes et l'{\' e}quation $\Delta u =e^u$},
J. de Math. Pures et Appl. {\bf 5}:4 (1898) 137-230.


  \bibitem{SCLX2017} J. Song, Y. Cheng, B. Li and B. Xu, \emph{Drawing cone spherical metrics via Strebel differentials}, arXiv:1708.06535 [math.CV]



 \bibitem{RT99} B. Russo and M. Teixidor i Bigas, Montserrat, \emph{On a conjecture of Lange}, J. Algebraic Geom. {\bf 8}:(3) (1999) 483-496.

   \bibitem{Tro86} M. Troyanov, \emph{Les surfaces euclidiennes {\' a} singularit{\' e}s coniques},  Enseign. Math. {\bf 32(2)}:1-2 (1986)  79-94.

   \bibitem{Tro89} M. Troyanov, \emph{Metrics of constant curvature on a sphere with two conical singularities}, Differential geometry (Pe${\rm \tilde{n}}${\' i}scola, 1988), 296-306, Lecture Notes in Math., {\bf 1410}, Springer, Berlin, 1989.

  \bibitem{Troyanov91} M. Troyanov, \emph{Prescribing curvature on compact surfaces with conical singularities}, Trans. Amer. Math. Soc. {\bf 324} (1991) 793-821.

  \bibitem{UY2000} M. Umehara and K. Yamada, \emph{Metrics of constant curveture $1$ with three conical singularities on the $2$-sphere}, Illinois J. Math. {\bf 44}:1 (2000) 72-94.

  \bibitem{Voisin02} C. Voisin, \emph{Hodge theory and complex algebraic geometry I}, Cambridge University Press, 2002. Translated from the French by Leila Schneps.


\end{thebibliography}
\end{document}